\documentclass[oneside,english]{amsart}
\usepackage[T1]{fontenc}
\usepackage[latin9]{inputenc}
\usepackage{amsthm}
\usepackage{amstext}
\usepackage{amssymb}
\usepackage{esint}

\makeatletter
\numberwithin{equation}{section} 
\numberwithin{figure}{section} 
  \theoremstyle{plain}
  \newtheorem*{thm*}{Theorem}
\theoremstyle{plain}
\newtheorem{thm}{Theorem}
  \theoremstyle{definition}
  \newtheorem{defn}[thm]{Definition}
  \theoremstyle{plain}
  \newtheorem{prop}[thm]{Proposition}
 \theoremstyle{definition}
  \newtheorem{example}[thm]{Example}
  \theoremstyle{plain}
  \newtheorem{lem}[thm]{Lemma}
  \theoremstyle{plain}
  \newtheorem{cor}[thm]{Corollary}

\newcommand{\ww}[1]{\mathbb{#1}}

\newcommand{\dd}[1]{\mbox{d}#1}
\newcommand{\germ}[1]{\ww{C}\left\{  #1\right\}  }
\newcommand{\pol}[2]{\ww{C}\left[#1\right]_{#2}}

\makeatother

\usepackage{babel}

\begin{document}
\global\long\def\ww#1{\mathbb{#1}}
\global\long\def\dd#1{\mbox{d}#1}
\global\long\def\germ#1{\ww C\left\{  #1\right\}  }

\title{Existence of non-algebraic singularities of differential equation}

\author{Yohann GENZMER, Loïc TEYSSIER}

\date{Juin 2009}

\maketitle
\begin{center}
{\footnotesize Laboratoire I.R.M.A.}\\
{\footnotesize{} Université de Strasbourg (France)}
\par\end{center}{\footnotesize \par}
\begin{abstract}
An algebraizable singularity is a germ of a singular holomorphic foliation
which can be defined in some local chart by a differential equation
with algebraic coefficients. We show that there exist at least countably
many saddle-node singularities of the complex plane that are not algebraizable.
\end{abstract}

\dedicatory{$\copyright$\textbf{2009, Journal of Differential Equations / Éditions
scientifiques et médicales Elsevier SAS.}}

~

We consider differential equations in the complex plane \begin{eqnarray}
A\left(x,y\right)\dd y & = & B\left(x,y\right)\dd x\label{eq:diff}\end{eqnarray}
near an isolated singularity, which can be conveniently located at
$\left(0,0\right)$ by translation. The coefficients $A$ and $B$
are germs of a holomorphic function with a common zero at $\left(0,0\right)$
and no common factor. We denote by $\lambda_{1}$ and $\lambda_{2}$
the eigenvalues of the linear part of the equation at $\left(0,0\right)$.
We will always assume that at least one of those is non-zero, say
$\lambda_{2}\neq0$, and set $\lambda:=\frac{\lambda_{1}}{\lambda_{2}}$.
We recall the following classical result :
\begin{thm*}
\textbf{\emph{(Poincaré and Dulac \cite{Dulac})}} If $\lambda\notin\ww R_{\leq0}$
then there exist two polynomials $P,\, Q$ such that the previous
differential equation is orbitally equivalent through a local analytic
change of coordinates to \begin{eqnarray*}
P\left(x,y\right)\dd y & = & Q\left(x,y\right)\dd x\,.\end{eqnarray*}
 If moreover $\lambda\notin\ww N\cup1/\ww N_{\neq0}$ then we can
choose $P\left(x,y\right)=x$ and $Q\left(x,y\right)=\lambda y$ (\emph{i.e.}
the equation is linearizable). 
\end{thm*}
We recall that two germs of differential equation are \emph{orbitally
equivalent} when there exists a germ of biholomorphism conjugating
their solutions. It thus turns out that a generic equation is orbitally
equivalent to a linear, or at least algebraic, equation. Up to now
an open question regarded whether \emph{every }differential equation
is algebraic in some local chart. Such an equation will be called
\emph{algebraizable}. Geometrically, it is equivalent to ask if any
germ of a singularity of foliation in the complex plane can be realized
as some singularity of a foliation of $\mathbb{CP}^{2}$. We aim to
prove that it is not so in the case of a saddle-node ($\lambda=0$),
as was expected in \cite{MR2} for non-linearizable resonant singularities.
Notice that these equations are nonetheless \emph{formally} algebraizable. 
\begin{thm}
There exist at least countably many non-equivalent saddle-node equations
\begin{eqnarray}
x^{2}\dd y & = & \left(y+\mbox{h.o.t}\right)\dd x\label{eq:sn_eqdiff}\end{eqnarray}
which are not algebraizable.
\end{thm}
Our proof is based on Martinet-Ramis' theorem about orbital classification
of such equations, stating that the equivalence classes of all equations
\eqref{eq:sn_eqdiff} under the action of local changes of coordinates
is in one-to-one correspondence with the space of germs $\germ h$.
We will give a more precise statement in Section~\ref{sec:analytic_MR}.
Our argument boils down to the following : since the space of orbitally
equivalent saddle-node equations is in one-to-one correspondence with
a functional space of germs and since this space is {}``big'' then
the trace of all algebraic equations should reasonably be {}``meagre'',
for instance in the sense of Baire. Many problems arise immediately,
one of those being that $\germ h$ cannot be endowed with a topology
which would make it a Baire space while at the same time preserving
the {}``nice'' structure of the set of algebraic equations. Another
problem lies in the fact that $\germ h$ might not be objectively
{}``big'' as it can be the range of a continuous map $\ww R\to\germ h$
and thus set-theoretically equivalent to the field of scalars. Hence
both set theory and topology are not sufficient to guarantee that
the heuristics works, and we must consider {}``analytic Baire properties''.
What makes things work is the fact that Martinet-Ramis' invariant
of classification is analytic with respect to the equation, as was
already known. The main part of our proof regarding this Baire analyticity
property deals with showing that Dulac's prenormalization procedure
is analytic too.\\

What is actually expected is that the \emph{typical }saddle-node equation
is non-algebraizable, \emph{i.e. }the set of non-algebraizable equations
is a $G_{\delta}$-dense subset of all saddle-node equations, not
only that the image of those non-algebraizable equations is a $G_{\delta}$-dense
subset of the space of invariants (which is what we prove here). To
do so one must consider a finer topology on spaces of germs than the
ones used presently and study analyticity and openness of maps from
and into these spaces. This requires a lot of additional technical
work and is currently being carried out. The authors nonetheless believe
this stronger result to be true.

\section{A topology on $\germ{z_{1},z_{2},\ldots,z_{n}}$}

In the sequel we use bold-typed letters to indicate multi-variables
$\mathbf{z}:=\left(z_{1},\ldots,z_{n}\right)\in\ww C^{n}$ or multi-indices
$\mathbf{J}:=\left(j_{1},\ldots,j_{n}\right)\in\ww N^{n}$. We use
the standard notations $\mathbf{J}!:=\prod_{\ell}\left(j_{\ell}!\right)$,
$\left|\mathbf{J}\right|:=\sum_{\ell}j_{\ell}$ and $\mathbf{z}^{\mathbf{J}}:=\prod_{\ell}z_{\ell}^{j_{\ell}}$.

\subsection{Norm on $\germ{\mathbf{z}}$}

~

Let us endow the topological space $\germ{\mathbf{z}}$ with the norm\[
\left|\left|f\right|\right|:=\sum_{\mathbf{J}}\frac{\left|a_{\mathbf{J}}\right|}{\mathbf{J}!}\]
where $f\left(\mathbf{z}\right)=\sum_{\mathbf{J}}a_{\mathbf{J}}\mathbf{z}^{\mathbf{J}}$.
Since the series $f$ is convergent $\left|\left|f\right|\right|$
is well defined and is a norm on the space $\germ{\mathbf{z}}$. Notice
that the space $\left(\germ{\mathbf{z}},\left|\left|.\right|\right|\right)$
is not complete since the sequence $\left(\sum_{\left|\mathbf{J}\right|\leq n}\sqrt{\mathbf{J}!}\mathbf{z}^{\mathbf{J}}\right)_{n\in\ww N}$
has the Cauchy property but is not convergent in the space of convergent
series. It is not even a Baire space. The evaluation $f\mapsto f\left(0\right)$
is continuous as well as the evaluation at $0$ of any derivative.
Hence, the family of projectors $J^{N}$ which associates to $f$
is $N$-jet is continuous.

\subsection{Analytical functions from \textmd{$\mathbb{C}^{p}$} to \textmd{$\germ{\mathbf{z}}$}}

~
\begin{defn}
Let $\Omega$ be a domain of $\ww C^{p}$ for $p\in\ww N_{\neq0}$.
\begin{enumerate}
\item A map $F\,:\,\Omega\to\germ{\mathbf{z}}$ is said to be \emph{strongly
analytic} if the map $\left(\mathbf{x},\mathbf{z}\right)\mapsto F\left(\mathbf{x}\right)\left(\mathbf{z}\right)$
is analytic with respect to the $n+p$ complex variables $x_{1},\cdots,x_{p}$\textbf{\emph{
}}\emph{and} $z_{1},\cdots,z_{n}$ on a neighbourhood of $\Omega\times\left\{ 0\right\} $. 
\item The map $F$ is said to be \emph{analytic} if for any point $\mathbf{x}$
in $\Omega$, there exists a linear map $L\,:\,\mathbb{C}^{p}\rightarrow\germ{\mathbf{z}}$
such that \[
F\left(\mathbf{x}+\mathbf{h}\right)=F\left(\mathbf{x}\right)+L\left(\mathbf{h}\right)+o\left(\mathbf{h}\right)\,.\]

\item A map $G:\ww C\left\{ \mathbf{w}\right\} \mapsto\ww C\left\{ \mathbf{z}\right\} $
is said to be strongly analytic if the image of any analytic family
of $\ww C\left\{ \mathbf{w}\right\} $ with a lower bounded radius
of convergence is an analytic family of $\ww C\left\{ \mathbf{z}\right\} $
with a lower bounded radius of convergence. 
\end{enumerate}
\end{defn}
\begin{prop}
If $F$ is strongly analytic then it is analytic.
\end{prop}
Notice that there exist analytic maps which are not strongly analytic:
the obstruction comes simply from the non-existence of local uniform
lower bound for the radius of convergence of series on any open ball
of $\germ{\mathbf{z}}$ for $\left|\left|.\right|\right|$. The following
example, due to J. Duval, illustrates that fact. 
\begin{example}
\label{exa:duval}Consider the family of compact sets for $\varepsilon>0$
\begin{eqnarray*}
K_{\varepsilon} & := & \ww D\backslash\left\{ 0<Im\left(z\right)<\varepsilon\right\} \end{eqnarray*}
which is the union of two simply connected, compact and connected
sets $K_{\varepsilon}^{+}$ and $K_{\varepsilon}^{-}$ such that,
say, $K_{\varepsilon}^{\pm}$ intersects $\pm i\ww R_{>0}$. According
to Runge's approximation theorem there exists a sequence of polynomials
$\left(P_{n}^{\varepsilon}\right)_{n\in\ww N}$ which is a uniform
approximation of the function defined by $x\in K_{\varepsilon}^{+}\mapsto\frac{1}{x}$
and $x\in K_{\varepsilon}^{-}\mapsto1$. There exists a slowly converging
sequence $\varepsilon_{n}>0$ such that $\sup_{x\in\ww D}\left|P_{n}^{\varepsilon_{n}}\left(x\right)\right|\leq\sqrt{n}$.
We now form the sequence $P_{n}:=P_{n}^{\varepsilon_{n}}$ and consider
the map:\begin{eqnarray*}
F\,:\, x\in\ww C & \mapsto & \sum_{j\in\ww N}P_{j}\left(x\right)^{j}z^{j}\,.\end{eqnarray*}
The reader can easily prove that $F\left(x\right)\in\germ z$ for
all $x\in\ww C$ and that its radius of convergence is $\left|x\right|$
if $Im\left(x\right)>0$ and equals $1$ otherwise. As a consequence
$F$ cannot be strongly analytic, as $\left(x,z\right)\mapsto F\left(x,z\right)$
is analytic on no neighbourhood of $\left(0,0\right)$, whereas $x\mapsto F\left(x\right)$
is analytic, for \begin{eqnarray*}
\left|\left|F\left(x+h\right)-F\left(x\right)-h\sum_{j\in\ww N}jP_{j}'\left(x\right)P_{j}\left(x\right)^{j-1}z^{j}\right|\right| & \leq & C\left|h\right|^{2}\sum_{j\in\ww N}\frac{\sqrt{j}^{j}}{j!}\end{eqnarray*}
if we require that $x$ belong to a smaller disc $r\ww D$, $0<r<1$,
thanks to Cauchy's formula as will be detailed further down.\end{example}
\begin{proof}
In the proof we assume that $n=p=1$ : the general case can be treated
in much the same way. Since analyticity is a local property, we can
also perform the proof in a neighbourhood of $0\in\mathbb{C}$. Let
us write $F\left(x\right)\left(z\right)=\sum_{j\geq0}f_{j}\left(x\right)z^{j}$.
Since $F\left(x\right)\left(z\right)$ is analytic as a map of two
variables, the series $F\left(x\right)$ are convergent on a common
open disc centered at $x=0$ with radius $2\rho$. The Cauchy formula
ensures that for any $j$\[
f_{j}\left(x\right)=\frac{\left(-1\right)^{j}}{2i\pi}\int_{\gamma}\frac{F\left(x\right)\left(\xi\right)}{\xi^{j+1}}d\xi\]
for any loop $\gamma$ in the disc of convergence. Substituting $\gamma:=\left\{ \left|\xi\right|=\rho\right\} $
yields\[
\rho^{j+1}\left|f_{j}\left(x\right)\right|\leq\left|\left|F\left(x\right)\right|\right|_{\infty,D\left(0,\rho\right)},\]
Since $F\left(x\right)\left(z\right)$ is bounded on $D\left(0,\beta\right)\times D\left(0,\rho\right)$
for some $\beta$, there exists a positive number $C$ such that for
any $j$ 

\[
\left|f_{j}\left(x\right)\right|\leq\frac{C}{\rho^{j}}.\]
Hence on a disc $D(0,\beta')$ with $\beta'<\beta$ we have a control
of the second derivative of the components of $f_{j}\left(x\right)$
\[
\left|f_{j}^{\left(2\right)}\left(x\right)\right|\leq\frac{C^{'}}{\rho^{j}}.\]
As a consequence, we have on a yet smaller disc :

\[
\left|f_{j}\left(x+h\right)-f_{j}\left(x\right)-hf_{j}^{(1)}\left(x\right)\right|\leq C^{''}\frac{1}{\rho^{j}}\left|h\right|^{2}\,.\]
Defining $D_{x}F\left(h\right)$ as $h\sum_{j\geq0}f_{j}^{(1)}\left(x\right)$,
which is a convergent series, yields \[
\left|\left|F\left(x+h\right)-F\left(x\right)-D_{x}F\left(h\right)\right|\right|\leq C^{''}e^{\frac{1}{\rho}}\left|h\right|^{2},\]
which ensures the analyticity of $F$.
\end{proof}

\section{Analytical Baire property of $\germ{\mathbf{z}}$}

We haven't been able to find a suitable {}``nice'' and reasonably
interesting topology on $\germ z$ in order to obtain a Baire space,
and surely it is not possible to do so if we agree on what {}``interesting
topology'' might be... We can prove that $\left(\germ z,\left|\left|\cdot\right|\right|\right)$
is not Baire. But we can also prove that this space cannot be covered
by countably many analytic subspaces, which is the purpose of this
paragraph.
\begin{defn}
~
\begin{enumerate}
\item An \emph{analytic subspace} of $\germ{\mathbf{z}}$ is the range of
an analytic map $F\,:\,\Omega\subset\ww C^{p}\to\germ{\mathbf{z}}$.
\item We say that $\ww C\left\{ \mathbf{z}\right\} $ is an \emph{analytic
Baire space} if it cannot be the union of a countable analytic subspaces.
\end{enumerate}
\end{defn}
Our main result is the following
\begin{thm}
\label{thm:analytic_baire}$\germ{\mathbf{z}}$ is an analytic Baire
space.
\end{thm}

\subsection{Annoying facts about $\germ z$}

~

We begin with proving the following 
\begin{lem}
$\germ z$ is in one-to-one correspondence with $\ww C$.
\end{lem}
This result is a consequence of the existence of a {}``Peano-curve''
in $\germ z$ for some relatively natural topology.
\begin{proof}
The space $\germ z$ is naturally a subset of $\ww C^{\ww N}$, which
can be endowed with the product topology. The induced topology on
$\germ z$ makes this space a connected and locally connected topological
space. Moreover for any $\left(p,r\right)\in\ww N\times\ww Q$ the
subset of $\germ z$ defined by \[
A_{p,r}:=\left\{ f\left(h\right)=\sum_{j\geq0}a_{j}h^{j}\,\,:\,\,\left|a_{j}\right|\leq pr^{j}\right\} \]
is compact. The union $\bigcup_{\ww N\times\ww Q}A_{p,r}$ covers
the whole $\germ z$, which means the latter is $\sigma$-compact
for the topology under consideration. A theorem of Hahn, Mazurkievicz,
Menger, Moore and Sirpie\'nski \cite{Sierpinski} states precisely
that the continuous images of $\left[0,1\right]$ are the compact,
connected and locally connected spaces. Therefore $\germ z$ is a
continuous image of $\mathbb{R}$, and obviously of $\ww C$, for
the above not-too-pathological product topology. A weaker consequence
is that from a purely set-theoretical point of view $\ww C$ and $\ww C\left\{ z\right\} $
are in one-to-one correspondence. 
\end{proof}
Now we show that 
\begin{lem}
$\left(\germ z,\left|\left|\cdot\right|\right|\right)$ is not a Baire
space.\end{lem}
\begin{proof}
We consider the following example due to R. Schäfke. Consider the
subspaces\begin{eqnarray*}
M_{N} & := & \left\{ \sum a_{j}z^{j}\,:\,\left|a_{j}\right|\leq N^{j}\right\} \,\,\,,\, N\in\ww N\,.\end{eqnarray*}
Obviously $\germ z=\cup_{N}M_{N}$. Moreover $M_{N}=\cap_{j}\left\{ \left|a_{j}\right|\leq N^{j}\right\} $
is closed as the association $f\mapsto f^{\left(j\right)}\left(0\right)$
is continuous, and its interior is empty as the example \ref{exa:duval}
shows that no neighbourhood of $f\in\germ z$ may admit a uniform
lower bound for the radius of convergence.
\end{proof}
As an inductive space $\mathbb{C}\left\{ z\right\} $ can also be
endowed with the inductive topology : this space becomes complete
but not Baire. In particular this topology cannot be induced by a
metric.

\subsection{Preliminaries}

~

In order to prove Theorem \ref{thm:analytic_baire} we will need to
eventually locate the proof within a Baire space to get a contradiction.
Let $\mathcal{A}$ be the subspace of $\mathbb{C}\left\{ z\right\} $
defined by\[
\mathcal{A}:=\left\{ f\left(z\right)=\sum_{j\geq0}a_{j}z^{j}\,:\,\left|a_{j}\right|\textup{ is bounded}\right\} \]
together with the norm $\left|\left|\cdot\right|\right|_{\infty}^{}$
:\[
\left|\left|f\right|\right|_{\infty}:=\sup_{j}\left|a_{j}\right|\,.\]
 $\left(\mathcal{A},\left|\left|\cdot\right|\right|_{\infty}\right)$
is a complete metric space and is thus a Baire space because it is
isometric to a subspace of $\mathbb{C}^{\mathbb{N}}$ formed by all
bounded sequences equipped with the sup-norm. 
\begin{lem}
\label{lem:still_closed}Let $S$ be a closed set in $\mathbb{C}\left\{ z\right\} $
for the norm $\left|\left|\cdot\right|\right|$. Then $S\cap\mathcal{A}$
is closed in $\mathcal{A}$ for the norm $\left|\left|\cdot\right|\right|_{\infty}^{}$.\end{lem}
\begin{proof}
Let $\left(f_{n}\right)$ be a sequence in $S\cap\mathcal{A}$ which
tends to $f$ when $n$ tends to infinity for the norm $\left|\left|\cdot\right|\right|_{\infty}$.
Then $f$ belongs to $\mathcal{A}$ since it is closed. Moreover,
as \[
\left|\left|f_{n}-f\right|\right|\leq e\left|\left|f_{n}-f\right|\right|_{\infty},\]
the sequence is convergent in $\mathbb{C}\left\{ z\right\} $ for
the norm $\left|\left|.\right|\right|$. Since $S$ is closed $f$
must belong to $S$ too.\end{proof}
\begin{lem}
\label{lem:freedom}A family $f_{1},\ldots,f_{n}\in\mathbb{C}\left\{ z\right\} $
is free over $\mathbb{C}$ if, and only if, there exists $p\in\mathbb{N}$
such that their $p$-jets are free over $\mathbb{C}$.\end{lem}
\begin{proof}
Suppose that for any $p\in\mathbb{N}$ there exists a non-trivial
relation \[
\Lambda_{p}:=\left(\lambda_{1,p},\ldots,\lambda_{n,p}\right)\neq0\]
 for the family $\varphi_{p}:=\left(J^{p}\left(f_{1}\right),\ \ldots,\ J^{p}\left(f_{n}\right)\right)$,
that is \begin{eqnarray*}
J^{p}\left(\sum_{j=1}^{n}\lambda_{j,p}f_{j}\right) & = & 0\,.\end{eqnarray*}
Up to rescalling $\Lambda_{p}$ one can suppose that it belongs to
the unit sphere of $\mathbb{C}^{n}$ and so consider some adherence
value $\left(\lambda_{1,\infty},\ldots,\lambda_{n,\infty}\right).$
Because if $J^{k+1}\left(f\right)=0$ then $J^{k}\left(f\right)=0$,
by taking the limit $p\to\infty$ while fixing an arbitrary $k$ we
obtain that $\Lambda_{\infty}$ is a non-trivial relation for $\varphi_{k}$
by continuity of $f\mapsto J^{k}\left(f\right)$, and thus is a non-trivial
relation for $\left(f_{1},\cdots,f_{n}\right)$. 
\end{proof}
\noindent According to this lemma, if $F$ is of maximal rank at
$x$, i.e. its rank is equal to the dimension of the source space,
there exists $N\in\mathbb{N}$ such that the function $J^{N}F$ is
of maximal rank. Since the space of polynomials of maximal degree
$N$ is of finite dimension, the function $J^{N}R$ is locally one-to-one
around $x$. So is the application $F$. Hence the
\begin{cor}
\label{cor:max_rank}Let $F:\Omega\subset\mathbb{C}^{n}\rightarrow\mathbb{C}\left\{ \mathbf{z}\right\} $. \end{cor}
\begin{enumerate}
\item If $D_{x}F$ is of rank $n$ then $F$ is locally one-to-one near
$x$. 
\item If $D_{x}F$ is of maximal rank $p<n$ then there exists a smooth
hypersurface $S$ of dimension $p$ at $x$ such that $\left.F\right|_{S}$is
of rank $p$ and has the same image as $F$. \end{enumerate}
\begin{proof}
The second part of the corollary is proved using the same result in
finite dimension : indeed, if the range of $F$ were some finite dimensional
vector space, one could choose for $S$ the hypersurface $\left\{ x_{i_{1}}=\cdots=x_{i_{n-p}}=0\right\} $
where $D_{\left(x_{j_{1}},\ldots,x_{j_{p}}\right)}F$ is of rank $p$
with $\left\{ 1,\cdots,n\right\} =\left\{ i_{1},\ldots,i_{n-p}\right\} \cup\left\{ j_{1},\ldots,j_{p}\right\} .$
Now if the range of $F$ were $\mathbb{C}\left\{ \mathbf{z}\right\} $,
one applies this argument to $J^{N}F$ for all $N$ big enough. 
\end{proof}
The key point to Theorem 3 is the following proposition :
\begin{prop}
\label{pro:empty_int}Let $F\,:\,\Omega\to\mathbb{C}\left\{ z\right\} $
be continuous, analytic and one-to-one on an open set $\Omega\subset\mathbb{C}^{n}$.
Let $E<\mathbb{C}\left\{ z\right\} $ be any subspace of infinite
dimension and suppose that $D_{x}F$ is of rank $n$ for some $x\in\Omega$.
Then there exist $\delta$ in $E$ and $\varepsilon>0$ such that
for any $0<\left|t\right|<\varepsilon$ the germ $F\left(x\right)+t\delta$
does not belong to $F\left(\Omega\right)$.\end{prop}
\begin{proof}
Suppose the claim is false and fix $\delta\in E\backslash\left\{ 0\right\} $.
There exists a sequence $\left(u_{n}\right)_{n\in\mathbb{N}}\subset\mathbb{C}^{n}$
such that, for $n$ large enough, $x+u_{n}\in\Omega$ and\[
F\left(x+u_{n}\right)=F\left(x\right)+\frac{\delta}{n}\,.\]
Any accumulation point $u$ of $u_{n}$ satisfies $F\left(x+u\right)=F\left(x\right)$.
Because $F$ is one-to-one $u$ must vanish, which in turn implies
that $u_{n}$ converges towards zero. Besides, the definition of differentiability
we use implies \[
o\left(u_{n}\right)=\left|\left|D_{x}F\left(u_{n}\right)-\frac{\delta}{n}\right|\right|=\left|\left|u_{n}\right|\right|\left|\left|D_{x}F\left(\frac{u_{n}}{\left|\left|u_{n}\right|\right|}\right)-\frac{\delta}{n\left|\left|u_{n}\right|\right|}\right|\right|.\]
Dividing by $\left|\left|u_{n}\right|\right|$ yields $\left|\left|D_{x}F\left(\frac{u_{n}}{\left|\left|u_{n}\right|\right|}\right)-\frac{\delta}{n\left|\left|u_{n}\right|\right|}\right|\right|=o\left(1\right).$
Now by compactness of the unit sphere of $\mathbb{C}^{n}$ we can
assume that $\frac{u_{n}}{\left|\left|u_{n}\right|\right|}$ tends
to some $u\neq0$ when $n$ tends to infinity. Hence $\frac{\delta}{n||u_{n}||}$
has to tend to some$\lambda\delta$ as $n$ tends to infinity and,
according to the rank assumption, $\lambda\neq0$. As a matter of
consequence

\[
D_{x}F\left(u\right)=\lambda\delta,\]
which cannot be possible for every $\delta$ in $E$, for the image
of the differential map $D_{x}F$ is finite dimensional.
\end{proof}

\subsection{Analytical Baire property of $\mathbb{C}\left\{ z\right\} $ : proof
of Theorem \ref{thm:analytic_baire}}

~

We show here that $\mathbb{C}\left\{ z\right\} $ has an analytical
Baire property by supposing on the contrary that $\mathbb{C}\left\{ z\right\} $
is a countable union of analytic sets : \[
\mathbb{C}\left\{ z\right\} =\bigcup_{n\in\mathbb{N}}\bigcup_{j\in\mathbb{N}}F_{j,n}\left(\Omega_{j,n}\right)\,,\]
where $F_{j,n}$ is a differentiable function defined on an open set
$\Omega_{j,n}$ of $\mathbb{C}^{n}$. Taking if necessary a finite
covering of each $\Omega_{j,n}$, one can assume that $F_{j,n}$ is
of rank $n$ on $\Omega_{j,n}$. Indeed the set of points where $F_{j,n}$
is not of maximal rank is an analytical subset $\Sigma_{j,n}$ of
$\Omega_{j,n}$ locally closed. The analytical set $\Sigma_{j,n}$
admits a decomposition $\Sigma_{j,n}=\cup C_{k}$ where each cell
$C_{k}$ is biholomorphic to an open set of some $\mathbb{C}^{p}$
with $0\leq p<n$ \cite{coste}. Hence we get the following decomposition\[
F_{j,n}\left(\Omega_{j,n}\right)=F_{j,n}\left(\Omega_{j,n}\backslash\Sigma_{j,n}\right)\bigcup_{k}F_{j,n}\left(C_{k}\right)\,.\]
If the rank $p$ of $F_{j,n}$ is strictly smaller than $n$ on $\Omega_{j,n}\backslash\Sigma_{j,n}$
then one can find a finer covering of $\Omega_{j,n}\backslash\Sigma_{j,n}=\bigcup_{k}B_{j,n,k}$
and a family of smooth hypersurfaces $S_{j,n,k}\subset B_{j,n,k}$of
dimension $p$ such that the rank of $\left.F_{j,n}\right|_{S_{j,n,k}}$is
$n$ and $\left.F_{j,n}\right|_{B_{j,n,k}}$ and $\left.F_{j,n}\right|_{S_{j,n,k}}$
has the same image . Now on each cell $C_{k}$ one can seek the points
where $\left.F_{j,n}\right|_{C_{k}}$ is not of maximal rank and do
the same procedure as above. This construction stops after finitely
many steps since at each stage the dimension of the open set we consider
is strictly less than that of the previous stage. Finally since any
open set of $\mathbb{C}^{p}$ is a countable union of compact sets,
we obtain the following decomposition\[
\mathbb{C}\left\{ h\right\} =\bigcup_{n\in\mathbb{N}}\bigcup_{j\in\mathbb{N}}\bigcup_{q\in\mathbb{N}}R_{j,n}\left(K_{j,n,q}\right)\,,\]
where $\Omega_{j,n}=\bigcup_{q\in\mathbb{N}}K_{j,n,q}$ and each $K_{j,n,q}$
is a full compact subset of some $\mathbb{C}^{p}$ with $p\leq n$.\\
The set $R_{j,n}\left(K_{j,n,q}\right)$ is compact and therefore
closed for the topology induced by $\left(\left|\left|\cdot\right|\right|_{k}\right)_{k>0}$.
According to Lemma \ref{lem:still_closed} the set $R_{j,n}\left(K_{j,n,q}\right)\cap\mathcal{A}$
is also closed in $\mathcal{A}$ for the norm $\left|\left|\cdot\right|\right|_{\infty}^{}$.
It is besides of empty interior : since $\mathcal{A}$ is infinite
dimensional if $R_{j,n}\left(x\right)$ belongs to $\mathcal{A}$
we can invoke Proposition \ref{pro:empty_int} to obtain $\delta\in\mathcal{A}$
such that for $t$ small enough\[
R_{j,n}\left(x\right)+t\delta\not\notin R_{j,n}\left(\Omega_{j,n}\right)\,,\]
which ensures that any small ball for the norm $\left|\left|\cdot\right|\right|_{\infty}^{}$
in $\mathcal{A}$ around $R_{j,n}\left(x\right)$ cannot be contained
in $R_{j,n}\left(\Omega_{j,n}\right)$. Finally we obtain the sought
contradiction since then $\mathcal{A}$ can be split into a countable
union of closed subset with empty interior :

\[
\mathcal{A}^{}=\bigcup_{n\in\mathbb{N}}\bigcup_{j\in\mathbb{N}}\bigcup_{q\in\mathbb{N}}R_{j,n}(K_{j,n,q})\cap\mathcal{A}\,,\]
which is impossible since $\mathcal{A}^{}$ is a Banach thus Baire
space.

\section{Analyticity of Martinet-Ramis invariants and proof of the main theorem\label{sec:analytic_MR}}

The tool that we need is a map which to a saddle-node equation, written
in the most general form \eqref{eq:diff}, associates its invariant
of orbital classification. This is the goal of this section, as well
as proving that this map is actually analytic.

Firstly we need to put the general equation \eqref{eq:diff} in a
prepared form; this is done using Dulac's prenormalization procedure
in Section~\ref{sub:Dulac's-procedure}. We will restrict our construction
to those equations whose first topological invariant equals $1$.
Geometrically speaking this invariant is the order of tangency between
the foliation defined by the equation and the separatrix tangent to
the eigenspace associated to the eigenvalue $\lambda_{2}\neq0$. This
defines the stratum $\mathcal{E}_{1}$, studied in Section~\ref{sub:The-stratum}.
After applying Dulac's procedure $\mathcal{D}$ we deal with equations
in the form\begin{eqnarray*}
x^{2}\mbox{d}y & = & \left(y+R\left(x,y\right)\right)\mbox{d}x\,.\end{eqnarray*}
Define the space $\mathcal{M}:=\ww C\times\ww C\times Diff\left(\ww C,0\right)$
the equivalence relation on $\mathcal{M}$ by $\left(\mu,\tau,\phi\right)\sim\left(\tilde{\mu},\tilde{\tau},\tilde{\phi}\right)$
if, and only if, $\mu=\tilde{\mu}$ and there exists $c\in\ww C_{\neq0}$
such that $\phi\left(ch\right)=\phi\tilde{f}\left(h\right)$ and $\tau=c\tilde{\tau}$. 
\begin{thm*}
\textbf{\emph{(Martinet-Ramis, \cite{MR2})}} There exists a map $M\,:\,\mathcal{E}_{1}\to\mathcal{M}$
such that two equations $E$ and $\tilde{E}$ of $\mathcal{E}_{1}$
are orbitally conjugate if, and only if, $M\left(E\right)=M\left(\tilde{E}\right)$.
Moreover this map is onto and if $\mathbf{t}\in\left(\ww C^{n},0\right)\mapsto E_{\mathbf{t}}\in\mathcal{E}_{1}$
is an analytic family of equations written in Dulac's form then $\mathbf{t}\mapsto\mathcal{M}\left(E_{\mathbf{t}}\right)$
is an analytic family too (that is, for all $\mathbf{t}$ one can
choose a representant of $\mathcal{M}\left(E_{\mathbf{t}}\right)$
such that this family is analytic).
\end{thm*}
In other words, once written in Dulac's form the germ-component of
Martinet-Ramis' map, which we will write $\phi_{MR}$, is strongly
analytic with respect to $R$. The aim of this section is to provide
a proof for the :
\begin{thm}
\label{thm:The-Martinet-Ramis-invariant}The complete Martinet-Ramis
map $\left(A\mbox{d}y-B\mbox{d}x\right)\overset{\mathcal{D}}{\longrightarrow}\left(x^{2}\mbox{d}y-\left(y+R\right)\mbox{d}x\right)\overset{\phi_{MR}}{\longrightarrow}\phi\in Diff\left(\ww C,0\right)$
is a strongly analytic association.
\end{thm}
Thus all that remains is to show that $\mathcal{D}$ is strongly analytic.
Before investigating this result we begin with giving the proof of
Theorem~1 in the upcoming section.

\subsection{\label{sub:Proof}Proof of the existence of countably many non-algebraizable
saddle-node singularities (Theorem 1)}

First we show that there there exists at least one such non-algebraizable
equation. Suppose on the contrary that any saddle-node singularity
is algebraizable. We define $\mathcal{P}_{d}$ to be the set of all
equations in $\mathcal{E}_{1}$ with polynomial coefficients of maximum
degree $d$ which, as we will see in Corollary~\ref{cor:p_n}, is
an analytic space. Then, according to Theorem\ref{thm:The-Martinet-Ramis-invariant}
the restriction of $\phi_{MR}\circ\mathcal{D}$ to the space of polynomials
must be onto : therefore, we would obtain the decomposition \[
Diff\left(\ww C,0\right)\simeq\mathbb{C}\left\{ h\right\} =\bigcup_{n\in\mathbb{N}}\phi_{MR}\circ\mathcal{D}\left(\mathcal{P}_{n}\right)\]
which would be a countable union of analytical sets. This is impossible
in view of the analytic Baire property of $\mathbb{C}\left\{ h\right\} $
(Theorem~\ref{thm:analytic_baire}). Hence, there exists at least
one saddle-node equation which is not algebraizable. 

Obviously the same argument works for countably many equations as
a point of $\germ h$ is a compact with empty interior.

\subsection{\label{sub:Dulac's-procedure}Dulac's procedure}

Let $\mathcal{E}$ be the set of couples $\left(A,B\right)\in\mathbb{C}\left\{ x,y\right\} \times\mathbb{C}\left\{ x,y\right\} $
such that the matrix \[
\left(\begin{array}{cc}
\frac{\partial A}{\partial x} & \frac{\partial A}{\partial y}\\
\frac{\partial B}{\partial x} & \frac{\partial B}{\partial y}\end{array}\right)\]
has exactly one non-vanishing eigenvalue. One can assume that, up
to a linear change of variables, the linear part of $X_{A,B}=-B\frac{\partial}{\partial x}+A\frac{\partial}{\partial y}$
is diagonal : \begin{eqnarray*}
A\left(x,y\right) & = & o\left(\left|\left|x,y\right|\right|\right)\\
B\left(x,y\right) & = & y+o\left(\left|\left|x,y\right|\right|\right)\,.\end{eqnarray*}
Notice that the latter change of variable depends rationally on the
coefficients of the linear part of $A$ and $B$. In all the sequel
the only changes of variables we allow will be required to preserve
this diagonal form. The existence of a unique analytic solution $x=s\left(y\right)$
(a separatrix of $X_{A,B}$) tangent to the eigenspace $\left\{ x=0\right\} $
at $\left(0,0\right)$ is well known (see \cite{CS} for example).
The other separatrix $y=\hat{s}\left(x\right)$, tangent to $\left\{ y=0\right\} $,
only exists \emph{a priori }at a formal level (and generically this
series, though unique, is divergent). The reader will find in \cite{Dulac2}
the material needed to carry out the complete prenormalization procedure.
What we retain from it is the following steps :
\begin{itemize}
\item Applying the change of coordinates $\left(x,y\right)\mapsto\left(x+s\left(y\right),y\right)$
transforms $X_{A,B}$ into a vector field $X_{A_{1},B_{1}}$ where
\begin{eqnarray*}
A_{1}\left(x,y\right) & \in & x\germ{x,y}\,\,.\end{eqnarray*}

\item It is possible to further orbitally normalize $A_{1}$ to obtain a
new vector field $X_{A_{D},B_{D}}$ such that \begin{eqnarray}
A_{D}\left(x,y\right) & = & x^{k+1}\nonumber \\
B_{D}\left(x,y\right) & = & y+r\left(x\right)+yR\left(x,y\right)\label{eq:dulac_prenorm_bis}\end{eqnarray}
with $r\left(0\right)=r'\left(0\right)=R\left(0,0\right)=0$. The
integer $k\in\ww N_{>0}$ is a topological invariant (but not a complete
topological invariant).
\item We define the map\begin{eqnarray*}
\mathcal{D}\,:\,\left(A,B\right)\in\mathcal{E} & \mapsto & B_{D}-y\in\germ{x,y}\,.\end{eqnarray*}
At this stage this map may not be well defined. We will give a canonical
way of obtaining $\mathcal{D}\left(A,B\right)$ from the original
vector field without ambiguity. We do this in Section~\eqref{sub:Dulac}.
\end{itemize}

\subsection{\label{sub:The-stratum}The stratum $\mathcal{E}_{1}$.}

Denote by $\mathcal{E}_{1}$ the stratum of $\mathcal{E}$ consisting
of equations that can be put under the previous form \eqref{eq:dulac_prenorm_bis}
with $k=1$. 
\begin{prop}
\label{pro:dulac_analytic}The stratum $\mathcal{E}_{1}$ is constructable:
it is the complementary of a dimension $1$ affine subspace of $\mathcal{E}$.\end{prop}
\begin{proof}
First we apply the change of coordinates $\left(x,y\right)\mapsto\left(x+s\left(y\right),y\right)$
which brings $X_{A,B}$ to $X_{x\tilde{A},\tilde{B}}$. In this situation
the separatrix is straightened to $\left\{ x=0\right\} $. Write $\tilde{A}\left(x,y\right)=ax+by+o\left(\left|\left|x,y\right|\right|\right)$;
we claim that $X_{A,B}$ belongs to $\mathcal{E}_{1}$ if, and only
if, $a\neq0$. On the one hand suppose that there exists a local analytic
change of coordinates $\Psi\left(x,y\right)=\left(\alpha x+C\left(x,y\right),\beta y+D\left(x,y\right)\right)$,
with $C$ and $D$ in $\germ{x,y}_{1}$, defining a conjugacy between
$X_{x\tilde{A},\tilde{B}}$ and some $UX_{x^{2},\hat{B}}$ with $\eta:=U\left(0,0\right)\neq0$.
Then :\begin{eqnarray}
U\left(x+C,y+D\right)\left(\alpha x+C\right)^{2} & = & x\tilde{A}\left(\alpha+\frac{\partial C}{\partial x}\right)+\tilde{B}\frac{\partial C}{\partial y}\,.\label{eq:Dulac_conj_x}\end{eqnarray}
Written for the term of least homogeneous degree this equation becomes,
since $\tilde{B}\left(x,y\right)=y+o\left(\left|\left|x,y\right|\right|\right)$
:\begin{eqnarray*}
\eta\alpha^{2}x^{2} & = & \alpha x\left(ax+by\right)+y\left(\delta x+\gamma y\right)\end{eqnarray*}
where $\delta=\frac{\partial^{2}C}{\partial x\partial y}\left(0,0\right)$
and $\gamma=\frac{1}{2}\frac{\partial^{2}C}{\partial y^{2}}\left(0,0\right)$.
Hence $\alpha\eta=a$, meaning $a\neq0$ as requested. On the other
hand we use Dulac's result : we know that there exists such a $\Psi$
between $X_{x\tilde{A},\tilde{B}}$ and some $UX_{x^{k+1},\hat{B}}$.
If $a\neq0$ then necessarily $k=1$, as can be seen for the analog
of \eqref{eq:Dulac_conj_x} (the term $\left(\tilde{B}-y\right)\frac{\partial C}{\partial y}$
is indeed of homogeneous degree strictly greater than $2$ and thus
cannot cancel $\alpha ax^{2}$ out). To complete the proof we only
have to mention that the condition $a\neq0$ is equivalent to $A_{2,0}\neq0$.
But this is obviously the case : we even have $A_{2,0}=a$ according
to \begin{eqnarray*}
A\left(x+s\left(y\right),y\right) & = & x\tilde{A}\left(x,y\right)+\tilde{B}\left(x,y\right)s'\left(y\right)\end{eqnarray*}
with $s'\left(0\right)=s\left(0\right)=0$. Hence $\mathcal{E}_{1}=\mathcal{E}\backslash\left\{ A_{2,0}=0\right\} $
is constructable. \end{proof}
\begin{cor}
\label{cor:p_n}Let $\pol{x,y}{\leq d}$ be the space of all polynomials
of degree at most $d$ and define\begin{eqnarray*}
\mathcal{P}_{d} & := & \mathcal{E}_{1}\cap\left(\pol{x,y}{\leq d}\times\pol{x,y}{\leq d}\right)\,.\end{eqnarray*}

Then $\mathcal{P}_{d}$ is a constructable set. 
\end{cor}
Particularly $\mathcal{P}_{d}$ is a finite union of smooth analytical
sets.

\subsection{\label{sub:Dulac}Strong analyticity of Dulac's procedure.}

We first begin with building the map $\mathcal{D}$ in a canonical
way. 
\begin{enumerate}
\item As already stated, there exists a unique germ $s\left(y\right)$ such
that $\left\{ x=s\left(y\right)\right\} $ is a separatrix of $X_{A,B}$.
\item Applying the change of coordinates $\left(x,y\right)\mapsto\left(x-s\left(y\right),y\right)$
transforms $X_{A,B}$ into $X_{A_{1},B_{1}}$ where \begin{eqnarray*}
B_{1}\left(x,y\right) & := & B\left(x-s\left(y\right),y\right)\\
A_{1}\left(x,y\right) & := & A\left(x-s\left(y\right),y\right)-B_{1}\left(x,y\right)s'\left(y\right)=:x\left(a_{0}\left(y\right)+\alpha xA_{2}\left(x,y\right)\right)\end{eqnarray*}
with $A_{2}\left(0,0\right)=1$ and $\alpha\neq0$.
\item There exists a unique holomorphic function $y\mapsto C\left(y\right)$
such that $C\left(0\right)=0$ and $\left(x,y\right)\mapsto\left(\frac{x}{\alpha}\left(1+C\left(y\right)\right),y\right)$
transforms $X_{A_{1},B_{1}}$ into $UX_{x^{2},B_{D}}$ where\begin{eqnarray*}
U\left(x,y\right) & := & \frac{a_{0}\left(y\right)}{B_{1}\left(0,y\right)}\frac{B_{1}\left(0,y\right)-B_{1}\left(x,y\right)}{\alpha x}+A_{2}\left(x,y\right)\\
B_{D}\left(x,y\right) & := & \frac{B_{1}\left(\frac{x}{\alpha}\left(1+C\left(y\right)\right),y\right)}{U\left(x,y\right)}.\end{eqnarray*}
This function $C$ is the unique holomorphic solution to the (regular)
linear differential equation with $C\left(0\right)=0$ :\begin{eqnarray*}
B_{1}\left(0,y\right)C'\left(y\right) & = & \left(1+C\left(y\right)\right)a_{0}\left(y\right),\end{eqnarray*}
whose solution is given by $C\left(y\right)=e^{\int_{0}^{y}\frac{a_{0}\left(u\right)}{B_{1}\left(0,u\right)}du}-1$.
We have $U\in\germ{x,y}^{*}$ since $U\left(0,0\right)=A_{2}\left(0,0\right)=1$. \end{enumerate}
\begin{defn}
We define Dulac's map as \begin{eqnarray*}
\mathcal{D}\left(A,B\right) & := & B_{D}-y\in\germ{x,y}_{1}.\end{eqnarray*}

\end{defn}
To prove that the map $\mathcal{D}$ is strongly analytic, it is enough
to prove that each step of the above construction shares this property:
it should be obvious for the second and third steps. It remains to
check that it is also the case for the first step. 
\begin{lem}
The correspondence $\left(A,B\right)\in\mathcal{E}_{1}\mapsto s\left(y\right)\in\mathbb{C}\left\{ z\right\} $
is strongly analytic.\end{lem}
\begin{proof}
It is enough to prove that one can control the disc of convergence
$s$ in terms of parameters depending on $A$ and $B$ . Let $\left(A_{\epsilon},B_{\epsilon}\right)$
be some analytic family in $\mathcal{E}_{1}$ with $\epsilon\in\left(\ww C^{p},0\right)$
The lemma is deduced from the following formal computation. Let us
write \[
s(y)=\sum_{j\geq0}s_{j}y^{j},\quad s_{0}=s_{1}=0\,.\]
Then for all $n\in\ww N$ :\[
s^{n}(y)=\sum_{j\geq0}\underbrace{\left(\sum_{j_{1}+\cdots+j_{n}=j}s_{j_{1}}\cdots s_{j_{n}}\right)}_{\mathcal{S}_{n,j}}y^{j}\]
where for $p\geq j$ we have $\mathcal{S}_{p,,j}=0$. Write \[
A_{\epsilon}\left(x,y\right)=\sum_{n,m}a_{n,m}^{\epsilon}x^{n}y^{m}\,\,\,\,,\, a_{0,0}^{\epsilon}=a_{1,0}^{\epsilon}=a_{0,1}^{\epsilon}=0\]
 so that \begin{eqnarray*}
A_{\epsilon}\left(s_{\epsilon}\left(y\right),y\right) & = & \sum_{n,m}a_{n,m}^{\epsilon}s_{\epsilon}(y)^{n}y^{m}=\sum_{p\geq0}\underbrace{\left(\sum_{j+m=p}\sum_{n\leq j}a_{n,m}^{\epsilon}\mathcal{S}_{n,,j}^{\epsilon}\right)}_{\mathcal{W}(A)_{p}}y^{p}\,.\end{eqnarray*}
The equation defining $s_{\epsilon}$, namely $A_{\epsilon}\left(s_{\epsilon}\left(y\right),y\right)=B_{\epsilon}\left(s_{\epsilon}\left(y\right),y\right)s_{\epsilon}^{'}\left(y\right)$,
thus becomes with a similar notation for $B_{\epsilon}\left(x,y\right)=\sum b_{n,m}^{\epsilon}x^{n}y^{m}$
: \[
\sum_{p\geq0}\mathcal{W}(A_{\epsilon})_{p}y^{p}=\sum_{p\geq0}ps_{p}^{\epsilon}y^{p}+\sum_{p\geq0}\left(\sum_{m+n-1=p}n\mathcal{W}(B_{\epsilon})_{m}s_{n}^{\epsilon}\right)y^{p}\,.\]
After identifying the coefficients in $y^{p}$ we derive\begin{equation}
ps_{p}^{\epsilon}=\mathcal{W}(A_{\epsilon})_{p}+\sum_{m+n=p+1}n\mathcal{W}(B_{\epsilon})_{m}s_{n}^{\epsilon}.\label{eq:reccu_sepx}\end{equation}
Hence, in a standard fashion, for any $p$ we have $\left|s_{p}^{\epsilon}\right|\leq\bar{s_{p}}$,
where $\bar{s_{p}}$ satisfies the same recurrence equation as $s_{p}^{\epsilon}$
except that we set $a_{m,n}=b_{m,n}=M\rho^{m+n}$ , where $M$ is
a constant and $\rho$ a lower bound for the radius of convergence
of $A_{\epsilon}$and $B_{\epsilon}$. Thus $\left|s_{p}^{\epsilon}\right|$
is less or equal than the coefficient $\bar{s}_{p}$ of $\bar{s}$
satisfying\[
\begin{array}{l}
\frac{1}{1-\rho y}\times\frac{1}{1-\rho\bar{s}\left(y\right)}-1-\rho y-\rho\bar{s}\left(y\right)=\bar{s}^{'}\left(y\right)\left(\frac{y}{M}+\frac{1}{1-\rho y}\times\frac{1}{1-\rho\bar{s}\left(y\right)}-1-\rho y-\rho\bar{s}\left(y\right)\right)\end{array}\]
 Since this equation admits a convergent solution with $\bar{s}\left(0\right)=0$,
its radius of convergence is a lower bound for the radius of convergence
of the family $s_{\epsilon}$. 
\end{proof}
\bibliographystyle{plain}
\bibliography{biblio}

\end{document}